\documentclass[12pt,reqno]{amsart}
\usepackage{amsfonts}
\usepackage{amssymb,color}
\usepackage{amssymb}

\def\normo#1{\left\|#1\right\|}

\def\aabs#1{\left|#1\right|}
\def\brk#1{\left(#1\right)}

\def\norm#1{\|#1\|}
\def\jb#1{\langle#1\rangle}

\newcommand{\R}{{\mathbb R}}
\newcommand{\C}{{\mathbb C}}
\newcommand{\Z}{{\mathbb Z}}
\newcommand{\ft}{{\mathcal{F}}}
\newcommand{\Hl}{{\mathcal{H}}}

\newcommand{\les}{{\lesssim}}
\newcommand{\ges}{{\gtrsim}}

\newcommand{\Lr}{{\mathcal{L}}}

\def\jb#1{\langle#1\rangle}
\def\norm#1{\|#1\|}
\def\normo#1{\left\|#1\right\|}

\def\aabs#1{\left|#1\right|}

\newcommand{\F}{\mathcal{F}}

\newcommand{\cP}{\mathcal{P}}

\newcommand{\cir}{\mathbb{S}}

\newcommand{\al}{\alpha}

\newcommand{\ga}{\gamma}

\newcommand{\e}{\varepsilon}

\newcommand{\om}{\omega}
\newcommand{\la}{\lambda}
\newcommand{\te}{\theta}

\newcommand{\x}{\xi}
\newcommand{\y}{\eta}

\newcommand{\ro}{\rho}

\newcommand{\ve}{\varepsilon}
\newcommand{\De}{\Delta}
\newcommand{\Des}{\Delta_\omega}
\newcommand{\Om}{\Omega}

\newcommand{\p}{\partial}
\newcommand{\na}{\nabla}
\newcommand{\re}{\mathop{\mathrm{Re}}}

\newcommand{\lec}{\lesssim}

\newcommand{\I}{\infty}

\newcommand{\LR}[1]{{\langle #1 \rangle}}

\newcommand{\EQ}[1]{\begin{equation}\begin{split} #1 \end{split}\end{equation}}
\newcommand{\Del}[1]{}
\newcommand{\CAS}[1]{\begin{cases} #1 \end{cases}}

\newcommand{\pt}{&}
\newcommand{\pr}{\\ &}
\newcommand{\pq}{\quad}

\numberwithin{equation}{section}

\newtheorem{thm}{Theorem}[section]
\newtheorem{cor}[thm]{Corollary}
\newtheorem{lem}[thm]{Lemma}

\theoremstyle{remark}
\newtheorem{rem}{Remark}

\theoremstyle{remark}

\theoremstyle{definition}

\begin{document}

\title[Zakharov system]{Generalized Strichartz estimates and scattering 
for 3D Zakharov system}
\author[Z. Guo]{Zihua Guo}
\address{LMAM, School of Mathematical Sciences, Peking
University, Beijing 100871, China}
\email{zihuaguo@math.pku.edu.cn}

\author[S. Lee]{Sanghyuk Lee}
\address{Department of Mathematical Sciences, Seoul National University, Seoul 151-747, Korea}
\email{shklee@snu.ac.kr}

\author[K. Nakanishi]{Kenji Nakanishi}
\address{Department of Mathematics, Kyoto University, Kyoto 606-8502,
Japan}
\email{n-kenji@math.kyoto-u.ac.jp}

\author[C. Wang]{Chengbo Wang}
\address{Department of Mathematics, Zhejiang University, Hangzhou 310027, China}
\email{wangcbo@gmail.com}

\begin{abstract}
We obtain scattering for the 3D Zakharov system with non-radial
small data in the energy space with angular regularity of degree
one. The main ingredient is a generalized Strichartz estimate for
the Schr\"odinger equation in the space of $L^2$ angular
integrability.

\end{abstract}

\maketitle

\section{Introduction}

The aim of this paper is twofold. We firstly obtain generalized
Strichartz estimates for radial dispersive equations. Secondly,
making use of these estimates we prove scattering for the 3D
Zakharov system with non-radial initial data.

\subsubsection*{Strichartz estimates} To begin with, let us consider the following
Schr\"odinger-type dispersive equations
\begin{align}\label{eq:schr}
i\partial_{t}u+D^a u=0,\ u(0,x)=f(x)
\end{align}
where $u(t,x):\R\times \R^{d}\to \C$, $D=\sqrt{-\Delta}$, $a>0$. Two
typical examples are the wave equation ($a=1$) and the Schr\"odinger
equation ($a=2$). The space time estimates which are called
Strichartz estimates address the estimates
\begin{align}\label{eq:striest}
\norm{e^{itD^a}P_0f}_{L_t^qL_x^p}\les \norm{f}_{L^2}
\end{align}
where $\widehat{P_0 f}\approx 1_{|\xi|\sim 1}\hat{f}$ (See the end
of this section for the precise definition).  Strichartz \cite{Str}
proved \eqref{eq:striest} for the case $q=p$ by drawing connection
between the estimate \eqref{eq:striest} and Fourier restriction
estimate which is known as Tomas-Stein theorem \cite{Tom} (also see
\cite[p.364-369]{Stein2} and references therein). Since then, the
estimates were substantially extended by various authors, e.g.
\cite{GiV95,LS} for $a=1$, and \cite{GV,Yaj} for $a=2$. It is now
well-known (see \cite{KT}) that \eqref{eq:striest} holds whenever
the following admissible condition is satisfied:
\begin{align}
\mbox{AP(a)}:
\begin{cases}
2\leq q,p\leq \infty,\frac{1}{q}\leq \frac{d-1}{2}(\frac{1}{2}-\frac{1}{p}), (q,p,d)\ne (2,\infty,3); \quad a=1,\\
2\leq q,p\leq \infty,\frac{1}{q}\leq
\frac{d}{2}(\frac{1}{2}-\frac{1}{p}), (q,p,d)\ne (2,\infty,2); \quad
a\neq 1.
\end{cases}
\end{align}

If the function $f$ is assumed to be radially symmetric, as expected
naturally \eqref{eq:striest} holds for a wider range of $(q,p)$ than
AP(a) (for example, see \cite{KlMa93, Sogge, FaWa2, Shao}). The optimal
range for \eqref{eq:striest}  under radial symmetry assumption  is
known except one endpoint case ($d\geq 2$). Namely,  if $f$ is
radial, then \eqref{eq:striest} holds if the following condition
holds:
\begin{align}
\mbox{RAP(a)}:
\begin{cases}
2\leq q,p\leq \infty,\frac{1}{q}<(d-1)(\frac{1}{2}-\frac{1}{p}) \mbox{ or } (q,p)=(\infty,2); \quad a=1,\\
2\leq q,p\leq \infty,\frac{1}{q}\leq
(d-\frac{1}{2})(\frac{1}{2}-\frac{1}{p}), (q,p)\ne
(2,\frac{4d-2}{2d-3}); \quad a\neq 1.
\end{cases}
\end{align}
The sharp range was obtained in \cite{GuoWang} except some endpoints
when $a>1$ and the remaining  endpoint estimates were later obtained
in \cite{CL, Ke} independently. See \cite{CL} for the estimates for
dispersive equations defined by more general pseudo-differential
operators.

Inspired by  the result for radial functions, one may try to  find a
weaker variant of Strichartz estimate which is valid for non-radial
functions and has the wider admissible range. There are two notable
approaches in this direction. The first  is to consider the estimate
with additional angular regularity
\begin{align}\label{eq:striestang}
\norm{e^{itD^a}P_0f}_{L_t^qL_x^p}\les \norm{f}_{H_\omega^{0,s}}
\end{align}
in which some angular regularity is traded off by the extension of
admissible range. See the end of this section for the definition of
$H_\omega^{0,s}$. When $a=1$, the estimate with almost sharp
regularity was obtained in \cite{Ster} (also see \cite{FaWa},
\cite{JWY} and \cite{CL}). When $a\neq 1$, in  \cite{CL} the authors
obtained \eqref{eq:striestang} for some $s>0$ and $(q,p)$ satisfying
RAP(a). However  the problem of obtaining \eqref{eq:striestang} with
the optimal regularity is still open. The latter
 is to consider the estimate with weaker angular
integrability, namely
\[
\norm{e^{itD^a}P_0f}_{L_t^q\Lr_{\rho}^pL_\omega^s}\les
\norm{f}_{L_x^2}
\]
for $s<p$.  Here the norm $L_t^q\Lr_{\rho}^pL_\omega^s$ for function
$u(t,x)$ on $\R\times \R^d (d\geq 2)$ is defined as follows
\[\norm{u}_{L_t^q\Lr_{\rho}^pL_\omega^s}=\brk{\int_\R\bigg[\int_0^\infty\bigg|\int_{\cir^{d-1}}
|u(t,\rho
x')|^sd\omega(x')\bigg|^{\frac{p}{s}}\rho^{d-1}d\rho\bigg]^{\frac{q}{p}}dt}^{1/q}.\]
Since $f$ is assumed to be in $L^2$ spaces, in view of orthogonality
of spherical harmonics the $L^2$ estimate
\begin{align}\label{eq:striestwea}
\norm{e^{itD^a}P_0f}_{L_t^q\Lr_{\rho}^pL_\omega^2}\les
\norm{f}_{L_x^2}
\end{align}
is most convenient to work with. These type of norm was used in
\cite{Tao2} to obtain the endpoint case of Strichartz estimate for
2D Schr\"odinger (see  \cite{MaNaNaOz05} for 3D wave equation) which
is not allowed in the usual  mixed norm spaces. For the wave
equation ($a=1$), it was known that \eqref{eq:striestwea} also holds
for RAP(1) pairs (see \cite{SmSoWa12, JWY, CL}). However, when
$a\neq 1$, as far as the authors know, it seems that
\eqref{eq:striestwea} is known only for some RAP(1) pair with
additional condition $q\geq p$ (see Theorem 1.7 in \cite{JWY}).

The first purpose of this paper is to consider the generalized
estimates of the type \eqref{eq:striestwea} for the case $a\neq 1$.
The other motivation is from the recent study of Zakharov system.
From the viewpoint of application, the estimate
\eqref{eq:striestwea} works better than \eqref{eq:striestang},
because there is no loss of angular regularity. The first result of
this paper is the following.

\begin{thm}\label{thm1}
Let $a>1, d\geq 3$ and $p(d)=\frac{6d-7+\sqrt{4d^2+4d-7}}{4d-7}$.
Then the estimate \eqref{eq:striestwea} holds if
\begin{align}\label{eq:Schrangular}
2\leq q,p\leq \infty,\quad \frac{1}{q}<
\frac{p(d)}{p(d)-2}(\frac{1}{2}-\frac{1}{p}) \mbox{ or }
(q,p)=(\infty,2).
\end{align}
\end{thm}

For the Schr\"odinger case $a=2$, our results are new. The range is
strictly contained in RAP(2), but wider than RAP(1) which is crucial
for the application to Zakharov system. For example, for $d=3$,
$p(3)\approx 3.48$, then $(2,4-)$ satisfies the condition.

Now we sketch the ideas in proving Theorem \ref{thm1}. 
We first use spherical harmonic expansion of $L^2(\cir^{d-1})$ to
reduce the estimates \eqref{eq:striestwea} to the uniform
boundedness of one-dimensional oscillatory integral operators
associated with Bessel functions of different orders.  To show
\eqref{eq:striestwea} we need to obtain the estimates which are
uniform along the orders of Bessel functions. For the wave case
$a=1$, uniform decay estimates of $J_\mu$ are sufficient. However,
to get \eqref{eq:striestwea} on the range wider than RAP(1) for the
case $a>1$, we need to exploit the oscillatory effect due to the
non-vanishing second derivative ($a>1$). For this purpose we split
Bessel function $J_\mu$ of order $\mu$ into two parts so that
$J_\mu=J_{\mu}^M+J_{\mu}^E$ (see \eqref{eq:Besselint}). For the
error term $J_{\mu}^E$, one has better uniform decay estimates, and
for the main term $J_{\mu}^M$ we basically rely on $TT^*$ method and
need to obtain uniform  kernel estimates for which  we  carry out
rather delicate analysis based on the stationary phase method.

\subsubsection*{Scattering for 3D Zakharov system} We now consider
the scattering problem for the 3D Zakharov system which  was
introduced by Zakharov \cite{Zak} as a mathematical model for the
Langmuir turbulence in unmagnetized ionized plasma:
\begin{align}\label{eq:Zak}
 \CAS{ i\dot u - \De u = nu,\\
   \ddot n/\al^2 - \De n = -\De|u|^2,}
\end{align}
with the initial data
\begin{align}
u(0,x)=u_0,\, n(0,x)=n_0,\,\dot n(0,x)=n_1,
\end{align}
where $(u,n)(t,x):\R^{1+3}\to\C\times\R$, and $\al>0$ denotes the
ion sound speed. It preserves $\|u(t)\|_{L^2_x}$ and the energy \EQ{
 E=\int_{\R^3}|\na u|^2+\frac{|D^{-1}\dot n|^2/\al^2+|n|^2}{2}-n|u|^2 dx}
The natural energy space for initial data is
\EQ{\label{eq:indata}
 (u_0,n_0,n_1)\in H^1(\R^3)\times L^2(\R^3)\times \dot H^{-1}(\R^3).}

Local wellposedness (without symmetry) was well understood. The
well-posedness in the energy space was proved in \cite{BoCo} for $d
= 2, 3$ and in \cite{GTV} for $d=1$, and in weighted Sobolev space
in \cite{KPV}. Unconditional uniqueness for 3D in energy space was
shown in \cite{MN2}. Improvement to the critical regularity was
obtained in \cite{GTV,BHHT} for $d=1,2$, and to the full subcritical
regularity in \cite{GTV,BeHe} for $d\geq 3$.  Global well-posedness
with small norm in energy space was proved in \cite{BoCo}. In
\cite{Takaoka,Kishi} the well-posedness for the system on the torus
was studied.  See \cite{SW,OT,MN} results on the subsonic limit to
NLS (as $\alpha\to \infty$). Concerning the long time and blow-up
behavior, Merle \cite{Merle} obtained blow-up in finite or infinite
time for negative energy, and the scattering was studied in
\cite{Shimo,GV2,OT2}.

Under radial symmetry  assumption  the scattering for 3D Zakharov
system with small energy was shown in \cite{GN} and global dynamics
below ground state was obtained in \cite{GNW}.  Very recently, the
scattering for non radial case was obtained by Hani-Pusateri-Shatah
\cite{HPS} under the assumption that the initial data are small
enough and have sufficient regularity and decay.

Here we take a different direction by considering  initial data with
additional angular regularity.  Compared with the condition imposed
for the scattering result in \cite{HPS}, our condition on the
initial data is much weaker in view of the simple embedding relation
between angular regularity and weighted Sobolev space. Moreover, our
result contains the existence of the wave operator, and hence the
scattering operator is constructed. The following is our second
result.

\begin{thm}\label{thm2}
Assume $\norm{(u_0,n_0,n_1)}_{H_\omega^{1,1}\times
H_\omega^{0,1}\times \dot H_\omega^{-1,1}}=\e$ for $\e>0$
sufficiently small. Then the global solution $(u,n)$ to
\eqref{eq:Zak} belongs to $C_t^0H_\omega^{1,1}\times
C_t^0H_\omega^{0,1}\cap C_t^1\dot{H}_\omega^{-1,1}$, and scatters in
this space: there exists $(u_{0,\pm},n_{0,\pm},n_{1,\pm})\in
H_\omega^{1,1}\times H_\omega^{0,1}\times \dot H_\omega^{-1,1}$ such
that
\begin{align}\label{eq:thmscat}
\lim_{t\to \pm
\infty}\norm{(u,n,\dot{n})-(u_\pm^l,n_\pm^l,\dot{n}_\pm^l)}_{H_\omega^{1,1}\times
H_\omega^{0,1}\times \dot H_\omega^{-1,1}}=0
\end{align}
where $(u_\pm^l,n_\pm^l)$ is the solution to the linear Zakharov
system with the initial datum $(u_{0,\pm},n_{0,\pm},n_{1,\pm})$.
Moreover, for any $(u_{0,\pm},n_{0,\pm},n_{1,\pm})\in
H_\omega^{1,1}\times H_\omega^{0,1}\times \dot H_\omega^{-1,1}$ with
sufficiently small norm, there exists a unique global solution
$(u,n)$ to \eqref{eq:Zak} such that \eqref{eq:thmscat} holds.
\end{thm}

We now give some words for the proof of Theorem \ref{thm2}.  The
main difficulty lies in derivative loss and slow dispersion of the
wave equation together with the quadratic nonlinearity. We basically
follows the idea in \cite{GN} by making use of the generalized
Strichartz estimates \eqref{eq:striestwea}. In fact, we combine
\eqref{eq:striestwea} and the normal form technique (for example,
see \cite{Shatah} and \cite{OTT}) to capture some nonlinear
oscillations. In order to get around the weak angular integrability,
we add some angular regularity so that the resulting space in
angular variable becomes a Banach algebra.


\subsubsection*{Notations} Finally we close this section by listing the notation.

 \noindent $\bullet$  $\ft(f)$ and $\widehat{f}$  denote
the Fourier transform of $f$. For $a\geq 1$,
$S_a(t)=e^{itD^a}=\ft^{-1}e^{it|\xi|^a}\ft$, and $S(t)=S_2(t)$.

\noindent $\bullet$ $\eta: \R\to [0, 1]$ is an even, non-negative
smooth function which is supported in $\{\xi:|\xi|\leq 8/5\}$ and
$\eta\equiv 1$ for $|\xi|\leq 5/4$.

\noindent $\bullet$ For $k\in \Z$
$\chi_k(\xi)=\eta(\xi/2^k)-\eta(\xi/2^{k-1})$ and $\chi_{\leq
k}(\xi)=\eta(\xi/2^k)$.

\noindent $\bullet$ $P_k, P_{\leq k}$ are defined on $L^2(\R^d)$ by
$\widehat{P_ku}(\xi)=\chi_k(|\xi|)\widehat{u}(\xi),\,\widehat{P_{\leq
k}u}(\xi)=\chi_{\leq k}(|\xi|)\widehat{u}(\xi)$.
\smallskip

\noindent $\bullet$   $\Des$ denotes the Laplace-Beltrami operator
on the unite sphere $\cir^{d-1}$ endowed with the standard metric
$g$ measure $d\omega$ and $\Lambda_\omega=\sqrt{1-\Delta_\omega}$.

\noindent $\bullet$  For $1\leq i,j\leq n$,
$X_{ij}=x_i\partial_j-x_j\partial_i$. It is well-known that for
$f\in C^2(\R^n)$
\[\Delta_\omega(f)(x)=\sum_{1\leq i<j\leq n}X_{ij}^2(f)(x).\]
Denote $L_\omega^p=L_\omega^p(\cir^{d-1})=L^p(\cir^{d-1}:d\omega)$,
$\Hl_p^s=\Hl_p^s(\cir^{d-1})=\Lambda_\omega ^{-s}L_\omega^p$.

\smallskip

\noindent $\bullet$ $L^p(\R^d)$ denotes the usual Lebesgue space,
and $\Lr^p(\R^+)=L^p(\R^+:r^{d-1}dr)$.

\noindent $\bullet$  $\Lr_r^pL_\omega^q $ and $\Lr_r^p\Hl^s_q$ are
Banach spaces defined  by the  following norms
\[\norm{f}_{\Lr_{r}^pL_\omega^q}=\big\|{\norm{f(r\omega)}_{L_\omega^q}}\big\|_{\Lr_{r}^p},\
\norm{f}_{\Lr_{r}^p\Hl^s_q}=\big\|{\norm{f(r\omega)}_{\Hl^s_q}}\big\|_{\Lr_{r}^p}.\]

\smallskip

\noindent $\bullet$  $H^s_p$, $\dot{H}_p^s$ ($B^s_{p,q}$,
$\dot{B}^s_{p,q}$) are the usual Sobolev (Besov) spaces on $\R^d$.

\noindent $\bullet$  $\dot{B}^s_{(p,q),r}$ denotes the Besov-type
space given  by the  norm
\[\norm{f}_{\dot{B}^s_{(p,q),r}}=(\sum_{k\in \Z}2^{ksr}\norm{P_kf}_{\Lr_r^pL_\omega^q}^r)^{1/r}.\]

\noindent $\bullet$  $H^{s,\alpha}_{p,\omega}$ is the space with the
norm
$\norm{f}_{H^{s,\alpha}_{p,\omega}}=\norm{\Lambda_\omega^{\alpha}f}_{H^s_p}$,
and the spaces $\dot{H}^{s,\alpha}_{p,\omega}$,
$B^{s,\alpha}_{p,q,\omega}$, $\dot{B}^{s,\alpha}_{p,q,\omega}$, and
$\dot{B}^{s,\alpha}_{(p,q),r,\omega}$ are defined similarly.

\noindent $\bullet$  For simplicity, we denote
$H^{s,\alpha}_{\omega}=H^{s,\alpha}_{2,\omega}$,
$\dot{H}^{s,\alpha}_{\omega}=\dot{H}^{s,\alpha}_{2,\omega}$,
$B^{s,\alpha}_{p,\omega}=B^{s,\alpha}_{p,2,\omega}$,
$\dot{B}^{s,\alpha}_{p,\omega}=\dot{B}^{s,\alpha}_{p,2,\omega}$,
$\dot{B}^{s,\alpha}_{(p,q),\omega}=\dot{B}^{s,\alpha}_{(p,q),2,\omega}$.

\smallskip

\noindent $\bullet$ Let $X$ be a Banach space on $\R^d$. $L_t^qX$
denotes the space-time space on $\R\times \R^d$ with the norm
$\norm{u}_{L_t^qX}=\big\|\norm{u(t,\cdot)}_X\big\|_{L_t^q}$.

\section{Generalized Strichartz Estimates}

In this section, we prove Theorem \ref{thm1}. First, we make some
reductions. To prove \eqref{eq:striestwea}, it is equivalent to show
\begin{align}\label{eq:Stri2}
\norm{T_af}_{L_t^q\Lr_{\rho}^pL_\omega^2}\les \norm{f}_{L_x^2},
\end{align}
where
\[T_af(t,x)=\int_{\R^d}e^{i(x\xi+t|\xi|^a)}\chi_0(|\xi|)f(\xi)d\xi.\]
Now we expand $f$ by the orthonormal basis $\{Y_k^l\}$, $k\geq
0,1\leq l\leq d(k)$ of spherical harmonics with
$d(k)=C_{n+k-1}^k-C_{n+k-3}^{k-2}$, such that
\[f(\xi)=f(\rho \sigma)=\sum_{k\geq 0}\sum_{1\leq l\leq d(k)}a_k^l(\rho)Y_k^l(\sigma).\]
Using the identities (see \cite{Stein1})
\[\widehat{Y_k^l}(\rho\sigma)=c_{d,k}\rho^{-\frac{d-2}{2}}J_\nu(\rho)Y_k^l(\sigma)\]
where $c_{d,k}=(2\pi)^{d/2}i^{-k}$, $\nu=\nu(k)=\frac{d-2+2k}{2}$,
then we get
\[T_af(t,x)=\sum_{k,l}c_{d,k}T_a^\nu (a_k^l)(t,|x|)Y_k^l(x/|x|),\]
where
\[T_a^\nu(h)(t,r)=r^{-\frac{d-2}{2}}\int e^{-it\rho^a}J_\nu(r\rho)\rho^{d/2}\chi_0(\rho)h(\rho)d\rho.\]
Here $J_\nu(r)$ is the Bessel function
\begin{align*}
J_\nu(r)=\frac{(r/2)^\nu}{\Gamma(\nu+1/2)\pi^{1/2}}
\int_{-1}^1e^{irt}(1-t^2)^{\nu-1/2}dt, \ \ \nu>-1/2.
\end{align*}
Thus \eqref{eq:Stri2} becomes
\begin{align}\label{eq:Stri3}
\norm{T_a^\nu (a_k^l)}_{L_t^q\Lr_{r}^pl_{k,l}^2}\les
\norm{\{a_k^l(\rho)\}}_{\Lr_\rho^2l_{k,l}^2}.
\end{align}
To prove \eqref{eq:Stri3}, it is equivalent to show
\begin{align}\label{eq:goal}
\norm{T_a^\nu (h)}_{L_t^q\Lr_{r}^p}\les \norm{h}_{L^2},
\end{align}
with a bound independent of $\nu$, since $q,p\geq 2$.

By the classical Strichartz estimates (see the endpoint estimates in
\cite{KT, MaNaNaOz05}
), we can get $\norm{1_{r\leq 100}T_a^\nu
(h)}_{L_t^q\Lr_{r}^p}\les \norm{h}_{L^2}$. Thus it remains to show
\begin{align}\label{eq:goal2}
\norm{1_{r\gg 1}T_a^\nu (h)}_{L_t^q\Lr_{r}^p}\les \norm{h}_{L^2},
\end{align}
with a bound independent of $\nu$. For any $R\gg 1$, define
\[S_R^{\nu,a}(h)(t,r)=\chi_0\big(\frac rR\big)\int e^{-it\rho^a}J_\nu(r\rho)\chi_0(\rho)h(\rho)d\rho.\]
Then
\[\norm{1_{r\gg 1}T_a^\nu}_{L^2\to L_t^q\Lr_{r}^p}\les \sum_{j\geq 5}2^{j(\frac{d-1}{p}-\frac{d-2}{2})}\norm{S_{2^j}^{\nu,a}}_{L^2\to L_t^qL_{r}^p}.\]
Then to prove \eqref{eq:goal2}, it suffices to show for some
$\delta>0$
\begin{align}\label{eq:goal3}
\norm{S_R^{\nu,a} (h)}_{L_t^qL_{r}^p}\leq C
R^{\frac{d-2}{2}-\frac{d-1}{p}-\delta}\norm{h}_{L^2},
\end{align}
where $C$ is independent of $\nu$. By interpolation, we only need to
show \eqref{eq:goal3} for $(q,p)=(2,p)$.

In the radial case, the estimates \eqref{eq:goal} can be reduced to
\eqref{eq:goal3} with $\nu=\frac{d-2}{2}$, see \cite{GuoWang}. The
same argument in \cite{GuoWang} also works for fixed $\nu$, but with
$\nu$-dependent bound, see also \cite{CL}. The difficulty in
\eqref{eq:goal3} is to obtain a uniform bound as $\nu\to \infty$,
thus the proof in \cite{GuoWang,CL} does not work. We need to
exploit the uniform properties of the Bessel function with respect
to $\nu$. In order to do so, we use the Schl\"{a}fli's integral
representation (see p. 176, \cite{Bess}):
\begin{align}\label{eq:Besselint}
J_\nu(r)=&\frac{1}{\pi}\re\int_0^\pi e^{i(r\sin \theta-\nu
\theta)}d\theta-\frac{\sin(\nu\pi)}{\pi}\int_0^\infty
e^{-\nu\tau-r\sinh \tau}d\tau \nonumber\\
:=&J_\nu^M(r)-J_\nu^E(r).
\end{align}
Obviously, the main term in \eqref{eq:Besselint} is $J_\nu^M(r)$. First we recall the Van der Corput Lemma (see p. 334,
\cite{Stein2}):
\begin{lem}[Van der Corput]\label{lem:staph}
Suppose $\phi$ is real-valued and smooth in $(a,b)$, and that
$|\phi^{(k)}(x)|\geq 1$ for all $x\in (a,b)$. Then
\[\aabs{\int_a^b e^{i\lambda \phi(x)}\psi(x)dx}\leq c_k \lambda^{-1/k}\bigg[|\psi(b)|+\int_a^b|\psi'(x)|dx\bigg]\]
holds when (i) $k\geq 2$, or (ii) $k=1$ and $\phi'(x)$ is monotonic.
Here $c_k$ is a constant depending only on $k$.
\end{lem}

Next, we recall the uniform decay estimates for the Bessel
functions, which can be found in \cite{Bess} p. 229 (7). For
completeness, we present a proof here because we need the estimate
for the error term.

\begin{lem}[Uniform decay estimates]\label{lem:Besseles} Assume $r,\nu>10$. Then we have
\begin{align}
\aabs{\int_0^\pi e^{i(r\sin \theta-\nu \theta)}d\theta}\leq&
C(1+|r^2-\nu^2|)^{-1/4},\label{ap-eq:besseldecay1}
\end{align}
As a consequence,
\begin{align}
|J_\nu(r)|+|J_\nu'(r)|\leq&
C(1+|r^2-\nu^2|)^{-1/4},\label{ap-eq:besseldecay3}
\end{align}
and for any $R\ges 1$,
\begin{equation}\label{eq-bessel-integral}
  \int_{r\sim R}|J_\nu(r)|^2+|J_\nu'(r)|^2dr\les 1.
\end{equation}
\end{lem}
\begin{proof}
We need only to show \eqref{ap-eq:besseldecay1}, since
\eqref{ap-eq:besseldecay3} follows immediately from
\eqref{ap-eq:besseldecay1}, the fact that $2
J_\nu'(r)=J_{\nu-1}(r)-J_{\nu+1}(r)$ (see p. 45 (2) in \cite{Bess})
and the following observation
\begin{align}\label{eq:besselerror}
|J_\nu^E(r)|+|(J_\nu^E)'(r)|\les (r+\nu)^{-1}.
\end{align}
We have
\[\aabs{\int_0^\pi e^{i(r\sin \theta-\nu \theta)}d\theta}\les \aabs{\int_0^{\pi/2} e^{i(r\sin \theta-\nu \theta)}d\theta}+\aabs{\int_{\pi/2}^\pi e^{i(r\sin \theta-\nu \theta)}d\theta}:=I+II.\]
Using Lemma \ref{lem:staph} and integration by part, we easily get
\begin{align}\label{eq:besselII}
II\leq& \aabs{\int_{\frac{\pi}{2}}^{\frac{\pi}{2}+1}e^{i(r\sin
\theta-\nu
\theta)}d\theta}+\aabs{\int_{\frac{\pi}{2}+1}^{\pi}e^{i(r\sin
\theta-\nu
\theta)}d\theta}\nonumber\\
\les& \min(r^{-1/2},\nu^{-1})+(1+r+\nu)^{-1}
\end{align}
which suffices to give the bound as desired. It remains to control
the term $I$. We have the trivial bound $I\les 1$. Then we show
$I\les |r^2-\nu^2|^{-1/4}$ case by case.

{\bf Case 1: $\nu>r$}. If $\nu\geq 2r$, then we have $I\les
\nu^{-1}$, which suffices to give the bound. Now we assume
$r<\nu<2r$. Fixing a $\ve>0$, then we get from the stationary phase
and integration by part that
\begin{align*}
I\les& \aabs{\int_{0}^{\ve} e^{i(r\sin \theta-\nu
\theta)}d\theta}+|\int_{\ve}^{\frac{\pi}{2}}e^{i(r\sin \theta-\nu
\theta)}d\theta|\\
\les& \min(\ve, |\nu-r|^{-1})+(r\ve)^{-1/2}\les
\ve^{1/2}|\nu-r|^{-1/2}+(r\ve)^{-1/2}.
\end{align*}
Thus setting $\ve=r^{-1/2}|r-\nu|^{1/2}$, we obtain the bound, as
desired.

{\bf Case 2: $\nu\leq r$}. We may assume further $\nu<r$. Let
$\theta_0=\arccos \frac{\nu}{r}$, then we have $\theta_0\sim
\sqrt{1-\frac{\nu^2}{r^2}}$. Denote $E=\{|\theta-\theta_0|\geq
\frac{\theta_0}{10}\}$ and $A=r^{-1/2}(r^2-\nu^2)^{1/8}$. Note that
$|r\cos\theta-\nu|\ges r\theta_0^2$ on $E\cap [0,\pi/2]$. If
$\theta_0\ges A$, then
\begin{align*}
I\les& \aabs{\int_{\{|\theta-\theta_0|<\frac{\theta_0}{10}\}\cap
[0,\pi/2]} e^{i(r\sin \theta-\nu \theta)}d\theta}+\aabs{\int_{E\cap
[0,\pi/2]} e^{i(r\sin
\theta-\nu \theta)}d\theta}\\
\les& (r\theta_0)^{-1/2}+(r\theta_0^2)^{-1}\les |r^2-\nu^2|^{-1/4},
\end{align*}
else if $\theta_0\ll A$, namely $r^{-1/2}(r^2-\nu^2)^{3/8}\ll 1$,
then
\begin{align*}
I\les& \aabs{\int_{E^c\cap [0,\pi/2]} e^{i(r\sin \theta-\nu
\theta)}d\theta}+\aabs{\int_{E\cap [0,A]} e^{i(r\sin \theta-\nu
\theta)}d\theta}+\aabs{\int_{E\cap [A,\pi/2]} e^{i(r\sin
\theta-\nu \theta)}d\theta}\\
\les& (r\theta_0)^{-1/2}+A+(rA^2)^{-1}\les |r^2-\nu^2|^{-1/4}.
\end{align*}
Thus the proof of \eqref{ap-eq:besseldecay1} is completed.
\end{proof}

We will also need to exploit the following asymptotical property of
the Bessel functions. The following lemma was obtained in
\cite{BC,BRV}, see e.g. Lemma 3 in \cite{BRV}.
\begin{lem}[Asymptotical property]\label{lem:Bessel}
Let $\nu>10$ and $r>\nu+\nu^{1/3}$, we have
\[J_\nu(r)=\frac{1}{\sqrt{2\pi}}\frac{e^{i\theta(r)}+e^{-i\theta(r)}}{(r^2-\nu^2)^{1/4}}+h(\nu,r),\]
where
\[\theta(r)=(r^2-\nu^2)^{1/2}-\nu \arccos \frac{\nu}{r}-\frac \pi 4\]
and
\[|h(\nu,r)|\les \bigg(\frac{\nu^2}{(r^2-\nu^2)^{7/4}}+\frac{1}{r}\bigg)1_{[\nu+\nu^{1/3},2\nu]}(r)+r^{-1}1_{[2\nu,\infty)}(r).\]
\end{lem}

With these uniform properties of Bessel function and the $L^2_t
L^2_r$ estimate, which is nothing but the local smoothing effect, we
are already able to obtain some new estimates.

\begin{lem}\label{prop:roughes}
Assume $a>0$. For $\nu>10$, $R\ges 1$, $2\leq p\leq \infty$
\begin{align}
\norm{S_R^{\nu,a} (h)}_{L_t^2 L_{r}^p}\les \norm{h}_{L^2}, \quad
\norm{S_R^{\nu,a} (h)}_{L_t^\infty L_{r}^2}\les \norm{h}_{L^2}.
\end{align}
\end{lem}
\begin{proof}
For the first inequality, we only need to prove the estimate for
$p=2,\infty$. For $p=2$, using Plancherel's equality in $t$ and Lemma \ref{lem:Besseles}, we get
\begin{align*}
\norm{\chi_0\big(\frac rR\big)\int
e^{-it\rho^a}J_\nu(r\rho)\chi_0(\rho)h(\rho)d\rho}_{L_{r}^2L_t^2}
\sim&\norm{\chi_0\big(\frac
rR\big)J_\nu(r\rho)\chi_0(\rho)h(\rho)}_{L_{r}^2L_\rho^2}\les
\norm{h}_2.
\end{align*}
Similarly, for $p=\infty$, we have
\begin{align*}
\norm{S_R^{\nu,a} (h)}_{L_t^2L_{r}^\infty} \les&\norm{S_R^{\nu,a}
(h)}_{L_t^2L_{r}^2}+\norm{\partial_r(S_R^{\nu,a}
(h))}_{L_t^2L_{r}^2}\les \norm{h}_2.
\end{align*}
The second inequality follows immediately from Minkowski's
inequality. Thus, the proof is completed.
\end{proof}

By the lemma above and interpolation, we see that \eqref{eq:goal3}
holds for RAP(1) pairs. In the rest of this section, we will refine
the estimate for $S_R^{\nu,a}$, assuming $a=2$, by making use of the
stationary phase method. The proof for the other case $1<a<2$ is
identical. For simplicity, we write $S_R^\nu=S_R^{\nu,2}$.

Fixing $\lambda\geq 100 R^{1/3}$, we decompose $S_R^\nu
(h)=\sum_{j=1}^3S_{R,j}^\nu (h)$, where
\[S_{R,j}^\nu (h)=\chi_0\big(\frac rR\big)\int
e^{-it\rho^2}J_\nu(r\rho)\gamma_j(\frac{r\rho-\nu}{\lambda})\chi_0(\rho)h(\rho)d\rho,\]
with $\gamma_1(x)=\eta(x)$, $\gamma_2(x)=(1-\eta(x))1_{x<0}$,
and $\gamma_3(x)=(1-\eta(x))1_{x>0}$. Then as $S_R^\nu$ we have
for $j=1,2,3$, $2\leq p\leq \infty$
\begin{align}
\norm{S_{R,j}^\nu (h)}_{L_t^2 L_{r}^p}\les \norm{h}_{L^2}.
\end{align}
We will refine the estimates for $S_{R,j}^\nu$, $j=1,2,3$,
respectively.

\begin{lem}\label{lem:SR} Assume $R\ges 1, \lambda\geq 100 R^{1/3}$, $2\leq p\leq \infty$. Then
\begin{align}
\norm{S_{R,1}^\nu (h)}_{L_t^2 L_{r}^p}\les&
\lambda^{1/4}R^{-1/4}\norm{h}_{L^2},\\
\norm{S_{R,2}^\nu (h)}_{L_t^2
L_{r}^p}\les& \big((\lambda^{-1}R^{1/4})^{1-\frac{2}{p}}+R^{-1/2}\big)\norm{h}_{L^2},\\
\norm{S_{R,3}^\nu (h)}_{L_t^2 L_{r}^p}\les&
\big(\lambda^{-\frac{1}{4}(1-\frac{2}{p})}+(\lambda^{-5/4}R^{1/4})^{2/p}+R^{-1/p}\big)\norm{h}_{L^2}.
\end{align}
\end{lem}

\begin{proof}
By interpolation, we only need to show the estimates for
$p=2,\infty$.

{\bf Step 1: estimate of $S_{R,1}^\nu$.}

As in the proof of Lemma \ref{prop:roughes}, by Lemma
\ref{lem:Besseles} we have
\begin{align*}
\norm{S_{R,1}^\nu (h)}_{L_t^2 L_{r}^2}\les
\norm{J_\nu(r)1_{|r-\nu|\les \lambda}}_2\norm{h}_2\les
\lambda^{1/4}R^{-1/4}\norm{h}_2.
\end{align*}
Similarly, we have
\begin{align*}
\norm{S_{R,1}^\nu (h)}_{L_t^2 L_{r}^\infty}\les
\lambda^{1/4}R^{-1/4}\norm{h}_2.
\end{align*}

{\bf Step 2: estimate of $S_{R,2}^\nu$.}

By the support of $\gamma_2$, we have $\nu>r\rho+\lambda$ in the
support of $\gamma_2(\frac{r\rho-\nu}{\lambda})$. Thus we use the
formula \eqref{eq:Besselint}. Without loss of generality, we assume
$J_\nu^M=\frac{1}{\pi}\int_0^{\pi}
 e^{i(r\rho\sin \theta-\nu
\theta)}d\theta$ (its conjugate part can be handled in the same
way), and decompose
\[S_{R,2}^\nu(h):=M_{R,2}^\nu(h)+E_{R,2}^\nu(h)\]
where
\begin{align*}
M_{R,2}^\nu (h)=&\chi_0\big(\frac rR\big)\int
e^{-it\rho^2}\int\bigg(\int_0^{\pi/2}
 e^{i(r\rho\sin \theta-\nu
\theta)}d\theta\bigg)\gamma_2(\frac{r\rho-\nu}{\lambda})\chi_0(\rho)h(\rho)d\rho,\\
E_{R,2}^\nu (h)=&\chi_0\big(\frac rR\big)\int
e^{-it\rho^2}\bigg(J_\nu^E(r\rho)+\int_{\pi/2}^{\pi}
 e^{i(r\rho\sin \theta-\nu
\theta)}d\theta\bigg)\gamma_2(\frac{r\rho-\nu}{\lambda})\chi_0(\rho)h(\rho)d\rho.
\end{align*}
By \eqref{eq:besselerror} and \eqref{eq:besselII}, we can easily get
for $2\leq p\leq \infty$
\begin{align*}
\norm{E_{R,2}^\nu (h)}_{L_t^2 L_{r}^p}\les R^{-1/2}\norm{h}_2.
\end{align*}

It remains to bound $M_{R,2}^\nu$. Denote
$\phi(r,\rho,\theta)=r\rho\sin \theta-\nu \theta$. Integrating by
part, we can decompose further
\begin{align*}
M_{R,2}^\nu(h)(t,r)=&\chi_0\big(\frac rR\big)\int\bigg(\frac
{e^{i\phi(r,\rho,\theta)}}{i(r\rho
\cos\theta-\nu)}\bigg|_{\theta=\pi/2}-\frac
{e^{i\phi(r,\rho,\theta)}}{i(r\rho
\cos\theta-\nu)}\bigg|_{\theta=0}\\
&-\int_{0}^{\pi/2}
 e^{i(r\rho\sin \theta-\nu
\theta)}\frac{r\rho \sin\theta}{(r\rho \cos
\theta-\nu)^2}d\theta\bigg)\gamma_2(\frac{r\rho-\nu}{\lambda})e^{-it\rho^2}\chi_0(\rho)h(\rho)d\rho\\
:=&M_{R,2,1}^\nu(h)-M_{R,2,2}^\nu(h)-M_{R,2,3}^\nu(h).
\end{align*} It suffices to prove:
for $R\ges 1$, $j=1,2,3$
\begin{align}
\norm{M_{R,2,j}^\nu (h)}_{L_t^2L_{r}^\infty}\les
\lambda^{-1}R^{1/4} \norm{h}_{L^2}.
\end{align}

Indeed, we have
\begin{align*}
M_{R,2,1}^\nu(h)=&\chi_0\big(\frac rR\big)\int\frac {e^{i(r\rho-\nu
\pi/2)}}{-i\nu}\gamma_2(\frac{r\rho-\nu}{\lambda})e^{-it\rho^2}\chi_0(\rho)h(\rho)d\rho,\\
M_{R,2,2}^\nu(h)=&\chi_0\big(\frac rR\big)\int\frac
{1}{i(r\rho-\nu)}\gamma_2(\frac{r\rho-\nu}{\lambda})e^{-it\rho^2}\chi_0(\rho)h(\rho)d\rho.
\end{align*}
First we consider $j=2$. By $TT^*$ argument, it is equivalent to
show
\[\norm{M_{R,2,2}^\nu(M_{R,2,2}^\nu)^*f}_{L_t^2L_{r}^\infty}\les \lambda^{-2}R^{1/2}\norm{f}_{L_t^2L_{r}^1}\]
where
\[M_{R,2,2}^\nu(M_{R,2,2}^\nu)^*f=\int K_2(t-t',r,r')f(t',r')dt'dr'\]
with the kernel
\[K_2(t,r,r')=\int e^{-it\rho^2}\frac
{\chi_0\big(\frac
rR\big)}{r\rho-\nu}\gamma_2(\frac{r\rho-\nu}{\lambda})\frac
{\chi_0\big(\frac
{r'}R\big)}{r'\rho-\nu}\gamma_2(\frac{r'\rho-\nu}{\lambda})\chi_0^2(\rho)d\rho.\]
It suffices to prove
\begin{align}\label{eq:K2est}
\norm{K_2}_{L_t^1L^\infty_{r,r'}}\les \lambda^{-2}R^{1/2}.
\end{align}
Denote $F_2=\frac {\chi_0\big(\frac
rR\big)}{i(r\rho-\nu)}\gamma_2(\frac{r\rho-\nu}{\lambda})\frac
{\chi_0\big(\frac
{r'}R\big)}{i(r'\rho-\nu)}\gamma_2(\frac{r'\rho-\nu}{\lambda})\chi_0^2(\rho)$.
First we have the trivial bound $|K_2|\les \lambda^{-2}$. Then by
Lemma \ref{lem:staph}, we get
\begin{align*}
|K_2(t,r,r')|\les& |t|^{-1/2} \int |\partial_\rho
F_2|d\rho\les\lambda^{-2}|t|^{-1/2}.
\end{align*}
On the other hand, if $|t|\gg R$, using integration by part twice,
get
\begin{align*}
|K_2|\les& \int
|t|^{-2}\aabs{\partial_\rho\big[\rho^{-1}\partial_\rho(\rho^{-1}F_2)\big]}d\rho\les
|t|^{-2}R\lambda^{-3}.
\end{align*}
Then eventually we have
\[|K_2|\les \lambda^{-2}|t|^{-1/2}1_{|t|\les R}+|t|^{-2}R\lambda^{-3}1_{|t|\gg R}\]
which implies the bound \eqref{eq:K2est} as desired.

For $j=1$, the proof follows in an similar and easier way as $j=2$
since $\nu\ges R$. Actually, we have
\begin{align}
\norm{M_{R,2,1}^\nu (h)}_{L_t^2L_{r}^\infty}\les
R^{-3/4}\norm{h}_{L^2}.
\end{align}
Now we consider $j=3$. Similarly, the kernel of
$M_{R,2,3}^\nu(M_{R,2,3}^\nu)^*$ is
\begin{align*}
K_3(t-t',r,r')=&\int_{0}^{\pi/2}\int_{0}^{\pi/2}\int
e^{-i(t-t')\rho^2}
 e^{i(r\rho\sin \theta-\nu
\theta)}\frac{\chi_0\big(\frac rR\big)r\rho \sin\theta}{(r\rho \cos
\theta-\nu)^2}
\gamma_2(\frac{r\rho-\nu}{\lambda})\\
&\times
 e^{-i(r'\rho\sin \theta'-\nu
\theta')}\frac{\chi_0\big(\frac {r'}R\big)r'\rho
\sin\theta'}{(r'\rho \cos
\theta'-\nu)^2}\gamma_2(\frac{r'\rho-\nu}{\lambda})\chi_0^2(\rho)d\rho
d\theta d\theta'.
\end{align*}
It suffices to prove
\[\norm{K_3}_{L_t^1L^\infty_{r,r'}}\les \lambda^{-2}R^{1/2}.\]

Denote $F_3(\rho)=\frac{\chi_0\big(\frac rR\big)r\rho
\sin\theta}{(r\rho \cos \theta-\nu)^2}\frac{\chi_0\big(\frac
{r'}R\big)r'\rho \sin\theta'}{(r'\rho \cos
\theta'-\nu)^2}\gamma_2(\frac{r\rho-\nu}{\lambda})\gamma_2(\frac{r'\rho-\nu}{\lambda})\chi_0^2(\rho)$.
First we have the trivial bound by integrating on $\theta,\theta'$
\begin{align*}
|K_3|\les &\int_{0}^{\pi/2}\int_{0}^{\pi/2}\int\frac{r\rho
\sin\theta\gamma_2(\frac{r\rho-\nu}{\lambda})}{(r\rho \cos
\theta-\nu)^2} \cdot\frac{r'\rho
\sin\theta'\gamma_2(\frac{r'\rho-\nu}{\lambda})}{(r'\rho \cos
\theta'-\nu)^2}\chi_0(\rho)d\rho d\theta d\theta'\les\lambda^{-2}.
\end{align*}
Then by Lemma \ref{lem:staph}, let \EQ{
 R_\te:=r\ro\cos\te-\nu, \pq R_\te':=r'\ro\cos\te'-\nu,}
we have
\begin{align*}
|K_3|\les &\int_{0}^{\pi/2}\int_{0}^{\pi/2} |t|^{-1/2}\int
|\partial_\rho
F_3|d\rho d\theta d\theta'\\
\les& |t|^{-1/2}\int_0^{\pi/2}\int_0^{\pi/2}\int |[\p_\ro\p_\te
R_\te^{-1}][\p_\te'(R_\te')^{-1}]|\ga_2(R_0/\la)\ga_2(R_0'/\la)\chi_0^2(\ro)d\ro
d\te d\te'\\
&+|t|^{-1/2}\int_0^{\pi/2}\int_0^{\pi/2}\int |[\p_\te
R_\te^{-1}][\p_\ro\p_\te'(R_\te')^{-1}]|\ga_2(R_0/\la)\ga_2(R_0'/\la)\chi_0^2(\ro)d\ro
d\te d\te'\\
&+ |t|^{-1/2}\int_0^{\pi/2}\int_0^{\pi/2}\int |[\p_\te
R_\te^{-1}][\p_\te'(R_\te')^{-1}]\p_\ro[\ga_2(R_0/\la)\ga_2(R_0'/\la)\chi_0^2(\ro)]|d\rho d\theta d\theta'\\
\les&|t|^{-1/2}\lambda^{-2},
\end{align*}
where we used the fact that $\p_\ro\p_\te R_\te^{-1}, \p_\te
R_\te^{-1}$ do not change the sign. On the other hand, for $|t|\gg
R$, denoting $\phi_1=-2t\rho+r\sin\theta-r'\sin \theta'$, we get
\begin{align*}
|K_3|\les &\int_{0}^{\pi/2}\int_{0}^{\pi/2}\int
|\partial_\rho\big(\phi_1^{-1}\partial_\rho[\phi_1^{-1}F]\big)|d\rho
d\theta d\theta'\les |t|^{-2}\lambda^{-3}R\ln(R).
\end{align*}
Then we get
\[|K_3|\leq \lambda^{-2}(1+|t|)^{-1/2}1_{|t|\les R}+|t|^{-2}\ln(R)\lambda^{-3}R1_{|t|\gg R}\]
which implies that $\norm{K_3}_{L_t^1L_{r,r'}^\infty}\les
\lambda^{-2}R^{1/2}$ as desired.

{\bf Step 3: estimate of $S_{R,3}^\nu$.}

By the support of $\gamma_3$, we have $r\rho>\nu+\lambda>\nu+\nu^{1/3}$
in the support of $\gamma_3(\frac{r\rho-\nu}{\lambda})$. Thus we use
the Lemma \ref{lem:Bessel}, and decompose
\[S_{R,3}^\nu(h):=M_{R,3}^\nu(h)+E_{R,3}^\nu(h)\]
where
\begin{align*}
M_{R,3}^\nu (h)=&\chi_0\big(\frac rR\big)\int
e^{-it\rho^2}\frac{e^{i\theta(r\rho)}+e^{-i\theta(r\rho)}}
{2 \sqrt{2\pi}(r^2\rho^2-\nu^2)^{1/4}}\gamma_3(\frac{r\rho-\nu}{\lambda})\chi_0(\rho)h(\rho)d\rho,\\
E_{R,3}^\nu (h)=&\chi_0\big(\frac rR\big)\int
e^{-it\rho^2}h(\nu,r\rho)\gamma_3(\frac{r\rho-\nu}{\lambda})\chi_0(\rho)h(\rho)d\rho,
\end{align*}
with $\theta(r),h(\nu,r)$ given in Lemma
\ref{lem:Bessel}.

First, we consider $M_{R,3}^\nu$. We only estimate
\[\widetilde{M}_{R,3}^\nu (h)=\chi_0\big(\frac rR\big)\int e^{-it\rho^2}\frac{e^{i\theta(r\rho)}}
{(r^2\rho^2-\nu^2)^{1/4}}\gamma_3(\frac{r\rho-\nu}{\lambda})\chi_0(\rho)h(\rho)d\rho,\]
since the other term is similar. It is easy to see
$\norm{\widetilde{M}_{R,3}^\nu (h)}_{L_t^2L_r^2}\les \norm{h}_2$, as
in the proof of Lemma \ref{prop:roughes}. We will prove
\[\norm{\widetilde{M}_{R,3}^\nu(h)}_{L_t^2L_r^\infty}\les \lambda^{-1/4}\norm{h}_{L^2}.\]
Similarly as in Step 2,
the kernel for $\widetilde{M}_{R,3}^\nu(\widetilde{M}_{R,3}^\nu)^*$
is
\[K(t-t',r,r')=\int e^{-i[(t-t')\rho^2-\theta(r\rho)+\theta(r'\rho)]}\frac{\chi_0\big(\frac rR\big)\gamma_3(\frac{r\rho-\nu}{\lambda})}
{(r^2\rho^2-\nu^2)^{1/4}}\frac{\chi_0\big(\frac
{r'}R\big)\gamma_3(\frac{r'\rho-\nu}{\lambda})}
{(r'^2\rho^2-\nu^2)^{1/4}}\chi_0^2(\rho)d\rho.\] It suffices to show
\[\norm{K}_{L_t^1L_{r,r'}^\infty}\les \lambda^{-1/2}.\]
Obviously, we have a trivial bound
\[|K|\les \lambda^{-1/2}R^{-1/2}.\]
Recall
$\theta(r)=(r^2-\nu^2)^{1/2}-\nu\arccos\frac{\nu}{r}-\frac{\pi}{4}$,
then direct computation shows
\begin{align*}
\theta'(r)=&(r^2-\nu^2)^{1/2}r^{-1},\\
\theta''(r)=&(r^2-\nu^2)^{-1/2}-(r^2-\nu^2)^{1/2}r^{-2}=(r^2-\nu^2)^{-1/2}\nu^2r^{-2},\\
\theta'''(r)=&(r^2-\nu^2)^{-3/2}\frac{\nu^2}{r}(-3+\frac{2\nu^2}{r^2}).
\end{align*}
Denoting $G=\frac{\chi_0\big(\frac
rR\big)\gamma_3(\frac{r\rho-\nu}{\lambda})}
{(r^2\rho^2-\nu^2)^{1/4}}\frac{\chi_0\big(\frac
{r'}R\big)\gamma_3(\frac{r'\rho-\nu}{\lambda})}
{(r'^2\rho^2-\nu^2)^{1/4}}\chi_0^2(\rho)$,
$\phi_2=t\rho^2-\theta(r\rho)+\theta(r'\rho)$. Then
\begin{align*}
\partial_\rho(\phi_2)=&2t\rho-\theta'(r\rho)r+\theta'(r'\rho)r'=2t\rho-
\frac{\rho(r'^2-r^2)}{\sqrt{r'^2\rho^2-\nu^2}+\sqrt{r^2\rho^2-\nu^2}}\\
\partial^2_\rho(\phi_2)=&2t-\theta''(r\rho)r^2+\theta''(r'\rho)r'^2\\
=&2t+
\frac{(r'^2-r^2)}{\sqrt{r'^2\rho^2-\nu^2}+\sqrt{r^2\rho^2-\nu^2}}
\frac{\nu^2}{\sqrt{r'^2\rho^2-\nu^2}\sqrt{r^2\rho^2-\nu^2}}\\
\partial^3_\rho(\phi_2)=&-\theta'''(r\rho)r^3+\theta'''(r'\rho)r'^3
\end{align*}
We observe that if $|\partial_\rho(\phi_2)|\ll |t|$, then
$|\partial^2_\rho(\phi_2)|\ges |t|$. This observation is also true
for the case $a>1$, but not true if $a<1$.

If $|t|\les R$, we divide $K$
\[K=\int e^{-i\phi_2}G \eta_0(\frac{100\partial_\rho(\phi_2)}{t})d\rho+
\int e^{-i\phi_2}G
[1-\eta_0(\frac{100\partial_\rho(\phi_2)}{t})]d\rho:=I_1+I_2.\] By
Lemma \ref{lem:staph}, we obtain
\begin{align*}
|I_1|\les& |t|^{-1/2}\bigg(\int |\partial_\rho G
\eta_0(\frac{100\partial_\rho(\phi_2)}{t})|d\rho+\int |G
\eta_0'(\frac{100\partial_\rho(\phi_2)}{t})\frac{100\partial_\rho^2(\phi_2)}{t}|d\rho\bigg)\\
\les& |t|^{-1/2}\lambda^{-1/2}R^{-1/2},
\end{align*}
where for the first term, we estimate it as $K_3$, while for the
second term, we only need to observe that
$\eta_0'(\frac{100\partial_\rho(\phi_2)}{t})\frac{100\partial_\rho^2(\phi_2)}{t}$
has fixed sign depending only on $t$.

For $I_2$, without loss of generality, we assume $r^2-r'^2>0$. Then
integrating by part, we get
\begin{align*}
|I_2|\les& \int \aabs{\partial_\rho \bigg((\partial_\rho\phi_2)^{-1}G[1-\eta_0(\frac{100\partial_\rho(\phi_2)}{t})]\bigg)}d\rho\\
\les& \int
\aabs{\frac{\partial^2_\rho\phi_2}{(\partial_\rho\phi_2)^2}G[1-\eta_0(\frac{100\partial_\rho(\phi_2)}{t})]}d\rho\\
&+|t|^{-1}(\int |\partial_\rho G|d\rho+\int |G
\eta_0'(\frac{100\partial_\rho(\phi_2)}{t})\frac{100\partial_\rho^2(\phi_2)}{t}|d\rho)\\
\les&\int
\frac{|\partial^2_\rho\phi_2|}{(\partial_\rho\phi_2)^2}G[1-\eta_0(\frac{100\partial_\rho(\phi_2)}{t})]d\rho+
|t|^{-1}\lambda^{-1/2}R^{-1/2}\\
\les&|t|^{-1}\lambda^{-1/2}R^{-1/2},
\end{align*}
where we used the fact that $\partial_\rho^2 \phi_2$ changes the
sign at most once.

If $|t|\gg R$, we have $|\partial_\rho(\phi_2)|\sim |t|$. Thus
integrating by part, we get
\begin{align*}
|K|\les& \int \aabs{\partial_\rho \big[(\partial_\rho\phi_2)^{-1}\partial_\rho \big((\partial_\rho\phi_2)^{-1}G\big)\big]}d\rho\\
\les& \int
\aabs{(\partial_\rho\phi_2)^{-3}\partial^3_\rho\phi_2G}d\rho+\int
\aabs{(\partial_\rho\phi_2)^{-2}\partial_\rho^2G}d\rho\\
&+\int
\aabs{(\partial_\rho\phi_2)^{-3}\partial^2_\rho\phi_2\partial_\rho
G}d\rho+\int
\aabs{(\partial_\rho\phi_2)^{-4}(\partial^2_\rho\phi_2)^2 G}d\rho\\
:=&II_1+II_2+II_3+II_4.
\end{align*}
As for $I_2$, we can obtain
\[II_2+II_3+II_4\les |t|^{-2}\lambda^{-1/2}R^{-1/2}R^2\lambda^{-2}\les |t|^{-2}\lambda^{-5/2}R^{3/2}.\]
For $II_1$, we have
\begin{align*}
II_1\les& |t|^{-3}\lambda^{-1/2}R^{-1/2}\int
(-\theta'''(r\rho)r^3-\theta'''(r'\rho)r'^3)
\gamma_3(\frac{r\rho-\nu}{\lambda})\gamma_3(\frac{r'\rho-\nu}{\lambda})d\rho\\
\les& |t|^{-3}\lambda^{-1/2}R^{-1/2}
\sup_{\rho:r\rho>\nu+\lambda}\theta''(r\rho)r^2\les
|t|^{-3}\lambda^{-1}R.
\end{align*}
Thus, eventually we get
\[|K|\les |t|^{-1/2}\lambda^{-1/2}R^{-1/2}1_{|t|\les R}+(|t|^{-2}\lambda^{-5/2}R^{3/2}+|t|^{-3}\lambda^{-1}R)1_{|t|\gg R}\]
which implies $\norm{K}_{L_t^1L_{r,r'}^\infty}\les \lambda^{-1/2}$
as desired, if $\lambda\ges R^{1/3}$.

It remains to bound $E_{R,3}^\nu$. First, we have for any $f\in L^2$
\[\norm{E_{R,3}^\nu(f)}_{L_t^2L_r^\infty}\les \norm{S_{R,3}^\nu(f)}_{L_t^2L_r^\infty}+\norm{M_{R,3}^\nu(f)}_{L_t^2L_r^\infty}\les \norm{f}_{L^2}.\]
On the other hand, using the decay estimate of $h(\nu,r)$, we get
\begin{align}\label{eq:ER3}
\norm{E_{R,3}^\nu(f)}_{L_t^2L_r^2}\les
(\lambda^{-5/4}R^{1/4}+R^{-1/2})\norm{f}_{L^2}.
\end{align}
Therefore, we complete the proof.
\end{proof}

Now we are ready to prove Theorem \ref{thm1}. At first, we notice that we need to assume $d\ge 3$, due to
the $R^{-1/p}$ bound which comes from the estimates of $E^\nu_{R,3}$ (see \eqref{eq:ER3}).
Next, we optimize the choice of $\lambda$. The main bounds in Lemma
\ref{lem:SR} are $\lambda^{1/4}R^{-1/4}$,
$\lambda^{-\frac{1}{4}(1-\frac{2}{p})}$. We can make them equal by
choosing
\[\lambda=R^{\frac{p}{2(p-1)}}\ges R^{1/3}.\]
Thus \eqref{eq:striestwea} holds, if $q=2$ and $p\geq 2$ satisfies
\[\frac{d-1}{p}-\frac{d-2}{2}+\frac{p}{8(p-1)}-\frac{1}{4}<0
\]
which is equivalent to
\[p>p(d)=\frac{6d-7+\sqrt{4d^2+4d-7}}{4d-7}.\]
Then by interpolation, we can obtain Theorem \ref{thm1}.

In order to apply these new Strichcartz estimates, we use
Christ-Kiselev lemma to derive the inhomogeneous linear estimates.
\begin{cor}\label{cor:str}
Assume $(q,p), (\tilde q,\tilde p)$ both satisfies
\eqref{eq:Schrangular}, $(q,p)\ne (\infty,2)$, and $q>\tilde q'$.
Then
\begin{align*}
\normo{\int_0^tS_a(t-s)P_0f(s)ds}_{L_t^q\Lr_{\rho}^pL_\omega^{2+}}\les
\norm{f}_{L_t^{\tilde q'}\Lr_{\rho}^{\tilde p'}L_\omega^2}
\end{align*}
\end{cor}
\begin{proof}
By Theorem \ref{thm1} and interpolation with \eqref{eq:striest} for
AP(a), we get slightly stronger estimates: for $(q,p)$ satisfy
\eqref{eq:Schrangular}, and $(q,p)\ne (\infty,2)$, then
\begin{align*}
\normo{S_a(t)P_0f}_{L_t^q\Lr_{\rho}^pL_\omega^{2+}}\les
\norm{f}_{L_x^2}.
\end{align*}
Thus this corollary following immediately from Christ-Kiselev lemma.
\end{proof}

\section{Scattering for Zakharov system}

This section is devoted to proving Theorem \ref{thm2}. The main
ingredients are the normal form reduction and the generalized
Strichartz estimates in Theorem \ref{thm1}.

\subsection{Normal form transform} We recall the normal form transform that was used in \cite{GN}.
It is convenient first to change the system into first order as
usual. Let \EQ{
 N:=n - iD^{-1}\dot n/\al,}
then $n=\re N=(N+\bar N)/2$ and the equations for $(u,N)$ are
\EQ{\label{eq:Zak1}
 \CAS{ (i\p_t-\De) u =Nu/2+\bar N u/2,\\
   (i\p_t+\al D) N = \al D|u|^2.}}
In our proof, the term $\bar N u$ makes no essential difference than
$Nu$, and hence for simplicity, we assume the nonlinear term in
first equation of \eqref{eq:Zak1} is $Nu$.

We use $S(t),W_\alpha(t)$ to denote the Schr\"odinger, wave
propagators:
\[S(t)\phi=\ft^{-1}e^{it|\xi|^2}\widehat{\phi},\quad W_\alpha(t)\phi=\ft^{-1}e^{i\alpha t|\xi|}\widehat{\phi}, \pq \widehat\phi=\F\phi.\]
For a quadratic term $uv$, we use $(uv)_{LH}$, $(uv)_{HH}$,
$(uv)_{HL}$ to denote the three different interactions
\[
 (uv)_{LH}=\sum_{k\in\Z}P_{\leq k-5}uP_kv,(uv)_{HL}=\sum_{k\in\Z}P_kuP_{\leq k-5}v,(uv)_{HH}=\sum_{\substack{|k_1-k_2|\leq 4 \\ k_1,k_2\in\Z}}P_{k_1}uP_{k_2}v.\]
To distinguish the resonant interaction, we also use \EQ{
 (uv)_{\al L}=\sum_{\substack{|k-\log_2\al|\le 1,\\ k\in\Z}}P_kuP_{\leq k-5}v,
 (uv)_{X L}=\sum_{\substack{|k-\log_2\al|> 1,\\ k\in\Z}}P_kuP_{\leq k-5}v,}
and similarly $(uv)_{L\al}$, $(uv)_{LX}$. It is obvious that we have
\EQ{
 uv\pt=(uv)_{HH}+(uv)_{LH}+(uv)_{HL}
 \pr=(uv)_{HH}+(uv)_{L\al}+(uv)_{LX}+(uv)_{\al L}+(uv)_{XL}.}
Moreover, for any such index $*=HH,HL,\al L$, etc., we denote the
bilinear symbol (multiplier) by \EQ{
 \F(uv)_* = \int \cP_*\widehat u(\x-\y)\widehat v(\y)d\y,}
and finite sum of those bilinear operators are denoted by the sum of
indices: \EQ{
 (uv)_{*_1+*_2+\cdots}=(uv)_{*_1}+(uv)_{*_2}+\cdots.}

By normal norm reduction (see \cite{GN}), we obtain that
\eqref{eq:Zak1} is equivalent to the following integral equation
\begin{align}
u=&S(t)u_0+S(t)\Om(N,u)(0)-\Om(N,u)(t)-i\alpha\int_0^tS(t-s)\Om(D|u|^2,u)(s)ds\nonumber\\
&-\frac{i}{2}\int_0^tS(t-s)
\Om(N,Nu)(s)ds-i\int_0^tS(t-s)(Nu)_{LH+HH+\alpha
L}ds,\label{eq:intu}\\
N=&W_\alpha(t)N_0+W_\alpha(t)D\tilde\Om(u,u)(0)-D\tilde\Om(u,u)(t)-\int_0^tW_\alpha(t-s)D(u\bar
u)_{HH+\alpha L+L\al}ds\nonumber\\
&-\int_0^tW_\alpha(t-s)(D\tilde\Om(Nu,u)+D\tilde\Om(u,Nu))(s)ds,\label{eq:intN}
\end{align}
where $\Om, \tilde \Om$ is a bilinear Fourier multiplier in the form
\begin{align*}
\Om(f,g)=&\F^{-1}\int\cP_{XL}\om^{-1}\widehat f(\x-\y)\widehat g(\y)d\y,\\
\tilde\Om(f,g)=&\F^{-1}\int \cP_{XL+LX}\frac{\widehat
f(\x-\y)\widehat{\bar g}(\y)}{|\xi-\eta|^2-|\eta|^2-\alpha|\xi|}d\y.
\end{align*}
The equations after normal form reduction look "roughly" \EQ{
 \pt(i\p_t+D^2)(u-\Om(N,u))=(Nu)_{LH+HH+\al L}+\Om(D|u|^2,u)+\Om(N,Nu),
 \pr(i\p_t+\al D)(N-D\tilde\Om(u,u))=D|u|^2_{HH+\al L+L\al}+D\tilde \Om(Nu,u)+D\tilde\Om(u,Nu).}

\subsection{The angular Strichartz space and main estimates}

Inspired by \cite{GN}, we introduce the Strichartz norm we need, and
present the main nonlinear estimates. For $u$ and $N$, we use the
Strichartz norms with angular regularity:
\begin{align}
u\in X=&\jb{D}^{-1}(L^\I_tH^{0,1}_\omega \cap L_t^2\dot
B^{1/4+\e,1}_{(q(\e),2+),\omega}\cap L_t^2\dot B^{0,1}_{6,\omega}),\label{Strz norms1}\\
N\in Y=&L^\I_tH^{0,1}_\omega \cap L^2_t\dot
B^{-1/4-\e,1}_{(q(-\e),2+),\omega},\label{Strz norms2}
\end{align}
for fixed $0<\e\ll 1$, where $q(\cdot)$ is defined by
\EQ{
 \frac{1}{q(\e)}=\frac{1}{4}+\frac{\e}{3}.}
The condition $0<\e\ll 1$ ensures that \EQ{
 \frac{7}{2}<q(\e)<4<q(-\e)<\I,}
such that the norms in \eqref{Strz norms1}-\eqref{Strz norms2}
satisfy the condition in Theorem \ref{thm1}.

We intend to apply the generalized Strichartz estimates to the
integral equations. To do the nonlinear estimates, we use some
representation theory of $SO(3)$. Let $\mu$ be the Haar measure of
$SO(3)$. In the sequel, we denote $L_A^q=L^q(SO(3),\mu)$.

\begin{lem}\label{lem:SO3}
(a) For any $1\leq p,q\leq \infty$, we have
\[\norm{f}_{\Lr_r^pL_\omega^q}\sim \norm{f(Ax)}_{L_x^pL_A^q}.\]

(b) For any $1<q<\infty$, we have
\[\norm{f}_{\Lr_r^p\Hl_q^1}\sim \norm{f}_{\Lr_r^pL_\omega^q}+\sum_{i,j}\norm{X_{i,j}f}_{\Lr_r^pL_\omega^q},\]
where $X_{i,j}=x_i\partial_j-x_j\partial_i$.
\end{lem}

\begin{proof}
Part (a) is direct. For the proof of part (b), see \cite{Taylor}.
\end{proof}

\begin{lem}\label{lem:conv}
Assume $1\leq p_1,p_2,q_1,q_2,p,q\leq \infty$,
$\frac{1}{q}=\frac{1}{q_1}+\frac{1}{q_2}$,
$1+\frac{1}{p}=\frac{1}{p_1}+\frac{1}{p_2}$. Then
\[\norm{f*g}_{\Lr_r^pL_\omega^q}\leq C \norm{f}_{\Lr_r^{p_1}L_\omega^{q_1}}\norm{g}_{\Lr_r^{p_2}L_\omega^{q_2}}.\]
\end{lem}
\begin{proof}
By Lemma \ref{lem:SO3}, we have
\begin{align*}
\norm{f*g}_{\Lr_r^pL_\omega^q}\sim&\normo{\int
f(Ax-y)g(y)dy}_{L_x^pL_A^q}\\
\sim&\normo{\int f(Ax-Ay)g(Ay)dy}_{L_x^pL_A^q}\les
\norm{f}_{\Lr_r^{p_1}L_\omega^{q_1}}\norm{g}_{\Lr_r^{p_2}L_\omega^{q_2}},
\end{align*}
where in the last inequality we used Minkowski's and H\"older's
inequalities in $A$, then Young's inequality in $x$.
\end{proof}

For a symbol $m$ on $\R^{6}$, define $T_m$ to be the bilinear
operator on $\R^3$:
\[T_m(f,g)(x)=\int_{\R^{6}} m(\xi,\eta)\widehat{f}(\xi)\widehat{g}(\eta)e^{ix(\xi+\eta)}d\xi d\eta.\]
The following lemma plays a crucial role.

\begin{lem}\label{lem:bilinear}
Let $1\leq p,p_1,p_2\leq \infty$ and $1/p=1/p_1+1/p_2$. Assume
$m(\xi,\eta)=h(|\xi|,|\eta|)$ for some function $h$, $m$ is bounded
and satisfies for all $\alpha,\beta$
\[|\partial_\xi^\alpha \partial_\eta^\beta m(\xi,\eta)|\leq C_{\alpha\beta}|\xi|^{-|\alpha|}|\eta|^{-|\beta|},\quad \xi,\eta \ne 0.\]
Then for $q>2$,
\[\norm{T_m(P_{k_1}f,P_{k_2}g)}_{\Lr_r^{p}\Hl^1_{q}}\leq C
\norm{f}_{\Lr_r^{p_1}\Hl^1_{q}}\norm{g}_{{\Lr_r^{p_2}\Hl^1_{q}}}\]
holds for any $k_1,k_2\in \Z$, with an uniform constant $C$.
\end{lem}

\begin{proof}
We can write
\[T_m(P_{k_1}f,P_{k_2}g)(x)=\int K(x-y,x-y')f(y)g(y')dydy'\]
where the kernel is given by
\[K(x,y)=\int m(\xi,\eta)\chi_{k_1}(\xi)\chi_{k_2}(\eta)e^{ix\xi+iy\eta}d\xi d\eta.\]
From the assumption on $m$, and integration by parts, we get a
pointwise bound of the kernel:
\begin{align}\label{eq:pointwiseK}
|K(x,y)|\les 2^{3k_1}(1+|2^{k_1}x|)^{-4}2^{3k_2}(1+|2^{k_2}y|)^{-4}.
\end{align}
Since $m$ is both radial in $\xi,\eta$, then $K$ is both radial in
$x,y$. Then by Lemma \ref{lem:SO3} (b) we get
\begin{align*}
\norm{T_m(P_{k_1}f,P_{k_2}g)}_{\Lr_r^{p}\Hl^1_{q}}\les&
\norm{T_m(P_{k_1}f,P_{k_2}g)}_{\Lr_r^{p}L^q_{\omega}}+\sum_{i,j}\norm{X_{i,j}T_m(P_{k_1}f,P_{k_2}g)}_{\Lr_r^{p}L^q_{\omega}}\\
:=&I+II.
\end{align*}
For the term $I$, by Minkowski's inequality and
\eqref{eq:pointwiseK}, we have
\begin{align*}
I\les&\normo{\int K(Ax-y,Ax-y')f(y)g(y')dydy'}_{L_x^pL_A^q}\\
\les&\normo{\int K(x-y,x-y')f(Ay)g(Ay')dydy'}_{L_x^pL_A^q}\\
\les&\norm{f(Ax)}_{L_x^{p_1}L_A^\infty}\cdot
\norm{g(Ax)}_{L_x^{p_2}L_A^q}\les\norm{f}_{\Lr_r^{p_1}L_\omega^{\infty}}\norm{g}_{\Lr_r^{p_2}L_\omega^{q}},
\end{align*}
then by the Sobolev embedding $\Hl_q^1(\cir^{2})\hookrightarrow
L_\omega^\infty$, this suffices to give the desired bound.
Similarly, for the term $II$, we have
\begin{align*}
II\les&\sum_{i,j}\norm{T_m(P_{k_1}X_{i,j}f,P_{k_2}g)}_{\Lr_r^{p}L^q_{\omega}}+\sum_{i,j}\norm{T_m(P_{k_1}f,P_{k_2}X_{i,j}g)}_{\Lr_r^{p}L^q_{\omega}}\\
\les&\sum_{i,j}
\norm{X_{i,j}f}_{\Lr_r^{p_1}L_\omega^{q}}\norm{g}_{\Lr_r^{p_2}L_\omega^{\infty}}+\sum_{i,j}
\norm{X_{i,j}g}_{\Lr_r^{p_1}L_\omega^{q}}\norm{f}_{\Lr_r^{p_2}L_\omega^{\infty}},
\end{align*}
which gives the desired bound, by Sobolev embedding.
\end{proof}

With these lemmas above, we can follow the proof with slight
modifications in \cite{GN} to prove the main estimates. The
following two lemmas can be proved similarly as Lemma 3.2-3.3 in
\cite{GN}. The main difference is that we use Lemma
\ref{lem:bilinear} for every bilinear dyadic piece.
\begin{lem}[Bilinear terms I]\label{lem:bi1}
(1) For any $N$ and $u$, we have
\begin{align*}
 \|(Nu)_{LH}\|_{L^1_tH^{1,1}_\omega} \les& \|N\|_{L^2_t\dot B^{-1/4-\e,1}_{(q(-\e),2+),\omega}}\|\jb{D}u\|_{L^2_t \dot
 B^{1/4+\e,1}_{(q(\e),2+),\omega}},\\
 \|(Nu)_{HH}\|_{L^1_tH^{1,1}_\omega} \les& \|N\|_{L^2_t\dot B^{-1/4-\e,1}_{(q(-\e),2+),\omega}}\|\jb{D}u\|_{L^2_t \dot
 B^{1/4+\e,1}_{(q(\e),2+),\omega}}.
\end{align*}

(2) If $0\leq \theta\leq 1$, $\frac{1}{\tilde
q}=\frac{1}{2}-\frac{\theta}{2}$, $\frac{1}{\tilde
r}=\frac{1}{4}+\frac{\theta}{3}+\frac{\e}{3}$, then for any $N$ and
$u$
\begin{align*}
\|(Nu)_{\alpha L}\|_{\LR{D}^{-1}L^{\tilde q'}_t \dot
B^{\frac{3}{2}-\frac{2}{\tilde q}-\frac{3}{\tilde r},1}_{(\tilde
r',2),\omega}} \lec \|N\|_{L^2_t\dot
B^{-1/4-\e,1}_{(q(-\e),2+),\omega}}\|u\|_{L_t^\infty
H_\omega^{0,1}\cap L^2_t \dot
 B^{1/4+\e,1}_{(q(\e),2+),\omega}}.
\end{align*}
\end{lem}

\begin{lem}[Bilinear terms II]\label{lem:bi2} (1) For any $u$, we have
\EQ{
 \|D(u\bar u)_{HH}\|_{L^1_tH^{0,1}_\omega} \lec \|u\|_{L^2_t\dot B^{1/4-\e,1}_{(q(-\e),2+),\omega}}\|\jb{D}u\|_{L^2_t \dot B^{1/4+\e,1}_{(q(\e),2+),\omega}}.}

(2) If $0\leq \theta\leq 1$, $\frac{1}{\tilde
q}=\frac{1}{2}-\frac{\theta}{2}$, $\frac{1}{\tilde
r}=\frac{1}{4}+\frac{\theta}{3}-\frac{\e}{3}$, then \EQ{
 \|D(u\bar u)_{\alpha L+L\al}\|_{L^{\tilde q'}_t \dot
B^{\frac{3}{2}-\frac{1}{\tilde q}-\frac{3}{\tilde r},1}_{(\tilde
r',2),\omega}} \lec \|\jb{D}u\|_{L_t^\infty H_\omega^{0,1}\cap L^2_t
\dot B^{1/4+\e,1}_{(q(\e),2+)},\omega}^2.}
\end{lem}

\begin{rem}
In application, we will use Lemma \ref{lem:bi1} (2) and Lemma
\ref{lem:bi2} (2) by fixing a $0<\theta_0\ll 1$ such that by this
choice $(\tilde q,\tilde r)$ is admissible to apply Corollary
\ref{cor:str}.
\end{rem}

Note the fact that $X_{i,j}(fg)=X_{i,j}fg+fX_{i,j}g$, and $X_{i,j}$
commute with the radial Fourier multiplier operator. Then the
following two lemmas are just the bilinear and trilinear estimates,
Lemma 3.5 and Lemma 3.7 obtained in \cite{GN} after applying
$X_{i,j}$, and not by Lemma \ref{lem:bilinear}.

\begin{lem}[Boundary terms]
(a) For any $N_0$ and $u_0$, we have
\begin{align*}
\|\Om(N_0,u_0)\|_{H^{1,1}_\omega} \les&
\|N_0\|_{H^{0,1}_\omega}\|u_0\|_{H^{1,1}_\omega},\\
\|D\tilde\Om(u_0,u_0)\|_{H^{0,1}_\omega}\les&
\|u_0\|_{H^{1,1}_\omega}\|u_0\|_{H^{1,1}_\omega}.
\end{align*}
As a consequence, for any $N$ and $u$
\begin{align*}
\|\Om(N,u)\|_{L_t^\I H^{1,1}_\omega} \les& \|N\|_{L_t^\I
H^{0,1}_\omega}\|u\|_{L_t^\I H^{1,1}_\omega},\\
\|D\tilde\Om(u,u)\|_{L_t^\I H^{0,1}_\omega}\les& \|u\|_{L_t^\I
H^{1,1}_\omega}\|u\|_{L_t^\I H^{1,1}_\omega}.
\end{align*}

(b) For any $N$ and $u$ we have
\begin{align*}
\|\LR{D}\Om(N,u)\|_{L^2_t\dot
B^{1/4+\e,1}_{q(\e),\omega}}\les&\|N\|_{L^\I_t
H^{0,1}_\omega}\|u\|_{L^2_t B^{1,1}_{6,\omega}},\\
\|D\tilde\Om(u,u)\|_{L^2_t \dot B^{-1/4-\e,1}_{q(-\e),\omega}}\les&
\|u\|_{L^2_tB^{1,1}_{6,\omega}}\|u\|_{L^\I_tH^{1,1}_\omega}.
\end{align*}
\end{lem}

\begin{lem}[Cubic terms] For any $N$ and $u$ we have
\begin{align*}
\|\Om(D|u|^2,u)\|_{L^1_tH^{1,1}_\omega}\les&
\|u\|_{L^2_t B^{1,1}_{6,\omega}}^2\|u\|_{L^\I_tH^{1,1}_\omega},\\
\|\LR{D}\Om(N,Nu)\|_{L^2_t\dot B^{0,1}_{6/5,\omega}}\les&
\|u\|_{L^2_t B^{1,1}_{6,\omega}}\|N\|^2_{L^\I_tH^{0,1}_\omega},\\
\|D\tilde\Om(Nu,u)\|_{L^1_tH^{0,1}_\omega}\les& \|u\|_{L^2_t
B^{1,1}_{6,\omega}}^2\|N\|_{L^\I_tH^{0,1}_\omega}.
\end{align*}
\end{lem}

With these estimates, we can prove the scattering part of Theorem
\ref{thm1}, see \cite{GN}. For the wave operator, see the appendix
of \cite{GNW}. We omit the details.

\subsection*{Acknowledgment}
The authors would like to thank Sebastian Herr for the helpful discussion.
Z. Guo is supported in part by NNSF of China (No. 11001003, No.
11271023). C. Wang is supported by Zhejiang Provincial Natural
Science Foundation of China LR12A01002, the Fundamental Research
Funds for the Central Universities (2012QNA3002), NSFC 11271322 and
J1210038.

\end{document}